\documentclass[12pt]{amsart}

\usepackage{amssymb}
\usepackage{bm}
\usepackage{graphicx}
\usepackage{psfrag}
\usepackage{color}
\usepackage{hyperref}
\hypersetup{colorlinks=true, linkcolor=blue, citecolor=magenta, urlcolor=wine}
\usepackage{url}
\usepackage{algpseudocode}
\usepackage{fancyhdr}
\usepackage{xy}
\input xy
\xyoption{all}

\newtheorem{theorem}{Theorem}[section]
\newtheorem{proposition}[theorem]{Proposition}
\newtheorem{lemma}[theorem]{Lemma}
\newtheorem{cor}[theorem]{Corollary}

\theoremstyle{definition}
\newtheorem{definition}[theorem]{Definition}
\newtheorem{example}[theorem]{Example}
\numberwithin{equation}{section}

\newcommand\nn{\mathbb{N}}
\newcommand\qq{\mathbb{Q}}
\newcommand\rr{\mathbb{R}}
\newcommand\zz{\mathbb{Z}}

\newcommand\pval{\mathsf{v}_p}

\keywords{Puiseux monoids, atomicity, prime-reciprocal monoids, cyclic rational semirings}

\subjclass[2010]{Primary: 13A05, 20M13; Secondary: 16Y60}

\begin{document}

	\mbox{}
	\title{Atomicity and boundedness of \\ monotone Puiseux monoids}
	\author{Felix Gotti}
	\address{Mathematics Department\\UC Berkeley\\Berkeley, CA 94720}
	\email{felixgotti@berkeley.edu}
	\author{Marly Gotti}
	\address{Mathematics Department\\University of Florida\\Gainesville, FL 32611}
	\email{marlycormar@ufl.edu}
	\date{\today}
	
	\begin{abstract}
		In this paper, we study the atomic structure of Puiseux monoids generated by monotone sequences. To understand this atomic structure, it is often useful to know whether the monoid has a bounded generating set. We provide necessary and sufficient conditions for the atomicity and boundedness to be transferred from a monotone Puiseux monoid to all its submonoids. Finally, we present two special subfamilies of monotone Puiseux monoids and fully classify their atomic structure.
	\end{abstract}
	
	\maketitle
	
	\section{Introduction} \label{sec:intro}
	
	Puiseux monoids were introduced in \cite{fG16}, where their atomic structure is studied. They are a natural generalization of numerical semigroups; however, while numerical semigroups are always atomic and minimally generated by their finite sets of atoms (irreducible elements), Puiseux monoids exhibit a very complex atomic structure. For instance, there are nontrivial Puiseux monoids having no atoms at all (i.e., being \emph{antimatter}), whereas others, failing to be atomic, contain infinitely many irreducible elements.
	
	Most of the Puiseux monoids whose sets of atoms have been determined can be ``nicely" generated, meaning that they contain generating sets with convenient properties such as finite, bounded, strongly bounded, etc. The simplicity of such generating sets allows us to have more control over the Puiseux monoid under study and, as a consequence, to better describe its atomic structure.
	
	In this paper, we will continue the study of the atomic structure of nicely generated Puiseux monoids, focusing now on those generated by a monotone sequence of rationals; we call them \emph{monotone} Puiseux monoids. Although the atomic behavior of this family will play the fundamental role here, we also study its boundedness. Even though boundedness does not seem to be related to atomicity \emph{a priori}, by imposing certain boundedness conditions we can control drastically the atomic structure. The following results, whose terminology will be recalled in the next section, shed light on this fact.
	
	\begin{theorem} \cite[Theorem 5.2]{fG16}
		Let $R = \{ r_n \mid n \in \nn \}$ be a strongly bounded subset of rationals generating the Puiseux monoid $M$. If $\mathsf{d}(r_n)$ divides $\mathsf{d}(r_{n+1})$, the sequence $\{\mathsf{d}(r_n)\}$ is unbounded, and $\mathsf{n}(R) \cap p\zz$ is finite for all prime $p$, then $M$ is antimatter.
	\end{theorem}
	
	\begin{theorem} \cite[Theorem 5.8]{fG16}
		Let $M$ be a strongly bounded finite Puiseux monoid. Then $M$ is atomic if and only if $M$ is isomorphic to a numerical semigroup.
	\end{theorem}
	
	Since every submonoid of a Puiseux monoid $M$ is again a Puiseux monoid, it is natural to ask whether a property of $M$ is inherited by all its submonoids. We say that a property $\mathsf{P}$ is \emph{hereditary} on $M$ if every submonoid of $M$ satisfies $\mathsf{P}$. Additionally, we say that a property $\mathsf{P}$ is hereditary on a class $\mathcal{C}$ of monoids if it is hereditary on every member of $\mathcal{C}$. As part of our study of monotone Puiseux monoids, we will find subfamilies where being atomic or monotone is hereditary.
	
	 In Section~\ref{sec:Background and Notation}, we establish the notation we will be using throughout this paper. In Section~\ref{sec:increasing PM}, we study the structure of Puiseux monoids that can be generated by increasing sequences. By contrast, Puiseux monoids that can be generated by decreasing sequences are investigated in Section~\ref{sec:decreasing PM}. Later, in Section~\ref{sec:primary Puiseux Monoids}, we focus on the study of a special class of decreasing Puiseux monoids whose members are precisely those generated by reciprocals of primes; we show that atomicity is hereditary on this class. Finally, in Section~\ref{sec:Multiplicative Cyclic Puiseux Monoids}, we describe the atomic structure of Puiseux monoids generated by geometric sequences (which are monotone), characterizing, in particular, those that are atomic.

	\medskip
	\section{Background and Notation} \label{sec:Background and Notation}
	
	To begin, we recall some basic definitions related to commutative semigroups and their sets of atoms. Reference material on commutative semigroups can be found in \cite{pG01} of Grillet. In addition, the monograph \cite{GH06} of Geroldinger and Halter-Koch offers extensive background information on atomic monoids and non-unique factorization theory, while \cite{aG09} of Geroldinger provides a more concise survey on the same area.
	
	The double-struck symbols $\mathbb{N}$ and $\mathbb{N}_0$ denote the sets of positive integers and non-negative integers, respectively. If $r$ is a real number, then we write $\zz_{\ge r}$ instead of $\{z \in \zz \mid z \ge r\}$; with a similar intention, we write $\qq_{\ge r}$ and $\qq_{> r}$. If $r \in \qq_{> 0}$, then the unique $a,b \in \nn$ such that $r = a/b$ and $\gcd(a,b)=1$ are denoted by $\mathsf{n}(r)$ and $\mathsf{d}(r)$, respectively. For $R \subseteq \qq_{>0}$, the sets $\mathsf{n}(R) = \{\mathsf{n}(r) \mid r \in R\}$ and $\mathsf{d}(R) = \{\mathsf{d}(r) \mid r \in R\}$ are called \emph{numerator set} and \emph{denominator set} of $R$, respectively.
	
	The unadorned term \emph{monoid} always means commutative cancellative monoid. If $M$ is a monoid, then because every monoid here is assumed to be commutative, we will use additive notation. In particular, $``+"$ denotes the operation of $M$, while $0$ denotes the identity element. We use the symbol $M^\bullet$ to denote the set $M \! \setminus \! \{0\}$. For $a,c \in M$, we say that $a$ \emph{divides} $c$ \emph{in} $M$ and write $a \mid_M c$ if $c = a + b$ for some $b \in M$. We write $M = \langle S \rangle$ when $M$ is generated by $S$, and say that $M$ is \emph{finitely generated} if it can be generated by a finite set.
	
	The set of units of $M$ is denoted by $M^\times$. An element $a \in M \! \setminus \! M^\times$ is \emph{irreducible} or an \emph{atom} if $a = u + v$ implies that either $u \in M^\times$ or $v \in M^\times$. We denote the set of atoms of $M$ by $\mathcal{A}(M)$. Every monoid $M$ in this paper will be \emph{reduced}, which means that $M^\times$ contains only the zero element. Therefore $\mathcal{A}(M)$ will be contained in each generating set. The monoid $M$ is \emph{atomic} if $M = \langle \mathcal{A}(M) \rangle$. On the other hand, if $\mathcal{A}(M)$ is empty, then we say that $M$ is \emph{antimatter}. Antimatter domains are defined in the same way as antimatter monoids are; they have been investigated by Coykendall et al. in \cite{CDM99}.
	
	A \emph{numerical semigroup} $N$ is a cofinite submonoid of the additive monoid $\nn_0$. Every numerical semigroup has a unique minimal set of generators, which happens to be finite. Additionally, if a numerical semigroup $N$ is minimally generated by positive integers $a_1 , \dots, a_n$, then $\gcd(a_1, \dots, a_n) = 1$ and $\mathcal{A}(N) = \{a_1, \dots, a_n\}$. Thus, every numerical semigroup is an atomic monoid containing finitely many atoms. A great introduction to the realm of numerical semigroups can be found in \cite{GR09} by Garc\'ia-S\'anchez and Rosales.
	
	A \emph{Puiseux monoid} is an additive submonoid of $\qq_{\ge 0}$. Albeit a natural generalization of numerical semigroups, Puiseux monoids are not always atomic. The following two results will be used throughout this paper without explicit mention.
	
	\begin{proposition}
		A Puiseux monoid $M$ is atomic if and only if it contains a minimal set of generators, which, in this case, must be $\mathcal{A}(M)$.
	\end{proposition}

	\begin{proof}
		Because Puiseux monoids are reduced, it follows from \cite[Proposition~1.1.7]{GH06}.
	\end{proof}

	\begin{proposition} \cite[Theorem~3.10]{fG16}
		If $0$ is not a limit point of a Puiseux monoid $M$, then $M$ is atomic.
	\end{proposition}
	
	The atomicity of Puiseux monoids was studied in \cite{fG16}, where the reader can find further results related to the atomic structure of these objects. A Puiseux monoid $M$ is said to be \emph{bounded} if $M$ can be generated by a bounded subset of rational numbers. There are Puiseux monoids that are not bounded, as we shall see below. Besides, we say that $M$ is \emph{strongly bounded} if $M$ can be generated by a set of rationals $R$ such that $\mathsf{n}(R)$ is bounded. Obviously, every finitely generated Puiesux monoid is strongly bounded. However, there are strongly bounded Puiseux monoids that fail to be finitely generated. The family of strongly bounded Puiseux monoids is strictly contained in that of bounded Puiseux monoids. The following example illustrates these observations.
	
	\begin{example}
		Let $P$ denote the set of primes. Consider the Puiseux monoid
		\[
			M_1 = \langle A_1 \rangle, \ \text{ where } \ A_1 = \bigg\{\frac{p^2-1}{p} \ \bigg{|} \ p \in P \bigg\}.
		\]
		An elementary divisibility argument will reveal that $\mathcal{A}(M_1) = A_1$. Because $A_1$ is an unbounded set, it follows that $M_1$ is not a bounded Puiseux monoid. On the other hand, let us consider the strongly bounded Puiseux monoid
		\[
			M_2 = \langle A_2 \rangle, \ \text { where } \ A_2 = \bigg\{\frac{1}{p} \ \bigg{|} \ p \in P \bigg\}.
		\]
		As in the previous case, it is not hard to verify that $\mathcal{A}(M_2) = A_2$. Therefore $M_2$ is a strongly bounded Puiseux monoid that is not finitely generated. We will say more about $M_2$ in Section~\ref{sec:primary Puiseux Monoids}. Finally, consider the bounded Puiseux monoid
		\[
			M_3 = \langle A_3 \rangle, \ \text{ where } \ A_3 = \bigg\{\frac{p-1}{p} \ \bigg{|} \ p \in P \bigg\}.
		\]
		Once again, $\mathcal{A}(M_3) = A_3$ follows from an elementary divisibility argument. As a result,~$M_3$ cannot be strongly bounded.
	\end{example}

	We conclude this section by recalling $p$-adic valuations and introducing the concept of finiteness for Puiseux monoids. Fix a prime number $p$. For a nonzero integer $a$, we define $\pval(a)$ to be the exponent of the maximal power of $p$ dividing $a$, and set $\pval(0) = \infty$. Moreover, for $b \in \zz \! \setminus \! \{0\}$, we define $\pval(a/b) := \pval(a) - \pval(b)$. The map $\pval$ is called the $p$-\emph{adic valuation}. It is not hard to verify that $\pval$ is \emph{semi-additive}, i.e.,
	\begin{equation*}
		\pval(r + s) \ge \min\{\pval(r), \pval(s) \} \text{ for all } \ r,s \in \qq^\times. \label{eq:semiadditivity of valuations}
	\end{equation*}
	Let $P$ be a set of primes. A Puiseux monoid $M$ \emph{over} $P$ is a Puiseux monoid such that $\pval(m) \ge 0$ for every $m \in M$ and $p \notin P$. If $P$ is finite, then we say that $M$ is a \emph{finite} Puiseux monoid over $P$. The Puiseux monoid $M$ is said to be \emph{finite} if there exists a finite set of primes $P$ such that $M$ is finite over $P$.

	\medskip
	\section{Increasing Puiseux Monoids} \label{sec:increasing PM}
	
	We are in a position now to begin our study of the atomic structure of Puiseux monoids generated by monotone sequences.
	
	\begin{definition}
		A Puiseux monoid $M$ is said to be \emph{increasing} (resp., \emph{decreasing}) if it can be generated by an increasing (resp., decreasing) sequence. A Puiseux monoid is \emph{monotone} if it is either increasing or decreasing.
	\end{definition}
	
	Not every Puiseux monoid is monotone, as the next example illustrates.
	
	\begin{example} \label{ex:bounded PM that is neither decreasing nor increasing}
		Let $p_1, p_2, \dots$ be an increasing enumeration of the set of prime numbers. Consider the Puiseux monoid $M = \langle A \cup B \rangle$, where
		\[
			A = \bigg\{ \frac{1}{p_{2n}} \ \bigg{|} \ n \in \nn \bigg\} \ \text{ and } \ B = \bigg\{ \frac{p_{2n-1} - 1}{p_{2n-1}} \ \bigg{|} \ n \in \nn \bigg\}.
		\]
		It follows immediately that both $A$ and $B$ belong to $\mathcal{A}(M)$. So $M$ is atomic, and $\mathcal{A}(M) = A \cup B$. Every generating set of $M$ must contain $A \cup B$ and so will have at least two limit points, namely, $0$ and $1$. Since every monotone sequence of rationals can have at most one limit point in the real line, we conclude that $M$ is not monotone.
	\end{example}
	
	The following proposition describes the atomic structure of the family of increasing Puiseux monoids.
	
	\begin{proposition} \label{prop:atoms of increasing monoids}
		Every increasing Puiseux monoid is atomic. Moreover, if $\{r_n\}$ is an increasing sequence of positive rationals generating a Puiseux monoid $M$, then $\mathcal{A}(M) = \{r_n \mid r_n \notin \langle r_1, \dots, r_{n-1} \rangle\}$.
	\end{proposition}
	
	\begin{proof}
		The fact that $M$ is atomic follows from observing that $r_1$ is a lower bound for~$M^\bullet$ and so $0$ is not a limit point of $M$. To prove the second statement, set
		\[
			A = \{r_n \mid r_n \notin \langle r_1, \dots, r_{n-1} \rangle\},
		\]
		and rename the elements of $A$ in a strictly increasing sequence (possibly finite), namely, $\{a_n\}$. Note that $M = \langle A \rangle$ and $a_n \notin \langle a_1, \dots, a_{n-1} \rangle$ for any $n \in \nn$. Since $a_1$ is the smallest nonzero element of $M$, we have that $a_1 \in \mathcal{A}(M)$. Suppose now that $n$ is a natural such that $2 \le n \le |A|$. Because $\{a_n\}$ is a strictly increasing sequence and $a_n \notin \langle a_1,\dots, a_{n-1} \rangle$, one finds that $a_n$ cannot be written as a sum of elements in $M$ in a non-trivial manner. Hence $a_n$ is an atom for every $n \in \nn$ and, therefore, we can conclude that $\mathcal{A}(M) = A$.
	\end{proof}
	
	Now we use Proposition~\ref{prop:atoms of increasing monoids} to show that every Puiseux monoid that is not isomorphic to a numerical semigroup has an atomic submonoid with infinitely many atoms. Let us first prove the next lemma.
	
	\begin{lemma} \label{lem:denominator set of like-NS PM}
		Let $M$ be a nontrivial Puiseux monoid. Then $\mathsf{d}(M^\bullet)$ is finite if and only if $M$ is isomorphic to a numerical semigroup.
	\end{lemma}
	
	\begin{proof}
		Suppose first that $\mathsf{d}(M^\bullet)$ is finite. Since $M$ is not trivial, $M^\bullet$ is not empty. Take $a \in \nn$ to be the least common multiple of $\mathsf{d}(M^\bullet)$. Since $aM$ is a submonoid of $\nn_0$, it is isomorphic to a numerical semigroup. Furthermore, the map $\varphi \colon M \to aM$ defined by $\varphi(x) = ax$ is a monoid isomorphism. Thus, $M$ is isomorphic to a numerical semigroup. Conversely, suppose that $M$ is isomorphic to a numerical semigroup. Because every numerical semigroup is finitely generated, so is $M$. Hence $\mathsf{d}(M^\bullet)$ is finite.
	\end{proof}
	
	\begin{proposition}
		If $M$ is a nontrivial Puiseux monoid, then it satisfies exactly one of the following conditions:
		\begin{enumerate}
			\item $M$ is isomorphic to a numerical semigroup;
			\vspace{3pt}
			\item $M$ contains an atomic submonoid with infinitely many atoms.
		\end{enumerate}
	\end{proposition}
	
	\begin{proof}
		Suppose that $M$ is not isomorphic to any numerical semigroup. Take $r_1 \in M^\bullet$\!\!. By Lemma~\ref{lem:denominator set of like-NS PM}, the set $\mathsf{d}(M^\bullet)$ is not finite. Therefore $\mathsf{d}(\langle r_1 \rangle^\bullet)$ is strictly contained in $\mathsf{d}(M^\bullet)$. Take $r'_2 \in M^\bullet$ such that $\mathsf{d}(r'_2) \notin \mathsf{d}(\langle r_1 \rangle^\bullet)$. Let $r_2$ be the sum of $m_2$ copies of~$r'_2$, where $m_2$ is a natural number so that $\gcd(m_2, \mathsf{d}(r'_2)) = 1$ and $m_2 r'_2 > r_1$. Setting $r_2 = m_2r'_2$, we notice that $r_2 > r_1$ and $r_2 \notin \langle r_1 \rangle$. Now suppose that $r_1, \dots, r_n \in M$ have been already chosen so that $r_{i+1} > r_i$ and $r_{i+1} \notin \langle r_1, \dots, r_i \rangle$ for $i = 1,\dots,n-1$. Once again, be using Lemma~\ref{lem:denominator set of like-NS PM} we can guarantee that $\mathsf{d}(M^\bullet)\setminus \mathsf{d}(\langle r_1, \dots, r_n \rangle^\bullet)$ is not empty. Take $r'_{n+1} \in M^\bullet$ such that $\mathsf{d}(r'_{n+1}) \notin \mathsf{d}(\langle r_1, \dots, r_n \rangle^\bullet)$, and choose $m_{n+1} \in \nn$ so that $\gcd(m_{n+1}, \mathsf{d}(r'_{n+1})) = 1$ and $m_{n+1} r'_{n+1} > r_n$. Taking $r_{n+1} = m_{n+1}r'_{n+1}$, one finds that $r_{n+1} > r_n$ and $r_{n+1} \notin \langle r_1, \dots, r_n \rangle$. Using the method just described, we obtain an infinite sequence $\{r_n\}$ of elements in $M$ satisfying that $r_{n+1} > r_n$ and $r_{n+1} \notin \langle r_1, \dots, r_n \rangle$ for every $n \in \nn$. By Proposition~\ref{prop:atoms of increasing monoids}, the submonoid $N = \langle r_n \mid n \in \nn \rangle$ is atomic and $\mathcal{A}(N) = \{r_n \mid n \in \nn\}$. Hence $M$ has an atomic submonoid with infinitely many atoms, namely, $N$.
		
		Finally, note that conditions (1) and (2) exclude each other; this is because a submonoid of a numerical semigroup is either trivial or isomorphic to a numerical semigroup and so it must contain only finitely many atoms.
	\end{proof}
	
	Now we split the family of increasing Puiseux monoids into two fundamental subfamilies. We will see that these two subfamilies have different behavior. We say that a sequence of rationals is \emph{strongly increasing} if it increases to infinity. On the other hand, a bounded increasing sequence of rationals is called \emph{weakly increasing}.
	
	\begin{definition}
		A Puiseux monoid is said to be \emph{strongly} (resp., \emph{weakly}) \emph{increasing} if it can be generated by a strongly (resp., weakly) increasing sequence.
	\end{definition}
	
	\begin{proposition}
		Every increasing Puiseux monoid is either strongly increasing or weakly increasing. A Puiseux monoid is both strongly and weakly increasing if and only if it is isomorphic to a numerical semigroup.
	\end{proposition}
	
	\begin{proof}
		The first statement follows straightforwardly. For the second statement, suppose that $M$ is a Puiseux monoid that is both strongly and weakly increasing. By Proposition~\ref{prop:atoms of increasing monoids}, the monoid $M$ is atomic, and its set of atoms can be listed increasingly. Let $\{a_n\}$ be an increasing sequence with underlying set $\mathcal{A}(M)$. Suppose, by way of contradiction, that $\mathcal{A}(M)$ is not finite. Since $M$ is strongly increasing, $\{a_n\}$ must be unbounded. However, the fact that $M$ is weakly decreasing forces $\{a_n\}$ to be bounded, which is a contradiction. Hence $\mathcal{A}(M)$ is finite, which implies that $M$ is isomorphic to a numerical semigroup.
		
		To prove the converse implication, take $M$ to be a Puiseux monoid isomorphic to a numerical semigroup. So $M$ is finitely generated, namely, $M = \langle r_1, \dots, r_n \rangle$ for some $n \in \nn$ and $r_1 < \dots < r_n$. The sequence $\{a_n\}$ defined by $a_k = r_k$ if $k \le n$ and $a_k = kr_n$ if $k > n$ is an unbounded increasing sequence generating $M$. Similarly, the sequence $\{b_n\}$ defined by $b_k = r_k$ if $k \le n$ and $b_k = r_n$ if $k > n$ is a bounded increasing sequence generating $M$. Consequently, $M$ is both strongly and weakly increasing.
	\end{proof}
	
	We will show that the strongly increasing property is hereditary on the class of strongly increasing Puiseux monoids. We will require the following lemma.
	
	\begin{lemma} \label{lem:no limit point implies strongly increasing subset}
		Let $R$ be an infinite subset of $\qq_{\ge 0}$. If $R$ does not contain any limit points, then it is the underlying set of a strongly increasing sequence.
	\end{lemma}
	
	\begin{proof}
		For every $r \in R$ and every subset $S$ of $R$, the interval $[0,r]$ must contain only finitely many elements of $S$; otherwise there would be a limit point of $S$ in $[0,r]$. Therefore every nonempty subset of $R$ has a minimum element. So the sequence $\{r_n\}$ recurrently defined by $r_1 = \min R$ and $r_n = \min R \! \setminus \! \{r_1, \dots, r_{n-1}\}$ is strictly increasing and has $R$ as its underlying set. Since $R$ is infinite and contains no limit points, the increasing sequence $\{r_n\}$ must be unbounded. Hence $R$ is the underlying set of the strongly increasing sequence $\{r_n\}$.
	\end{proof}
	
	\begin{theorem} \label{thm:strongly increasing iff super increasing}
		A nontrivial Puiseux monoid $M$ is strongly increasing if and only if every submonoid of $M$ is increasing.
	\end{theorem}
	
	\begin{proof}
		If $M$ is finitely generated, then it is isomorphic to a numerical semigroup, and the statement of the theorem follows immediately. So we will assume for the rest of this proof that $M$ is not finitely generated. Suppose that $M$ is strongly increasing. Let us start by verifying that $M$ does not have any real limit points. By Proposition~\ref{prop:atoms of increasing monoids}, the monoid $M$ is atomic. As $M$ is atomic and non-finitely generated, $|\mathcal{A}(M)| = \infty$. Let $\{a_n\}$ be an increasing sequence with underlying set $\mathcal{A}(M)$. Since $M$ is strongly increasing and $\mathcal{A}(M)$ is an infinite subset contained in every generating set of $M$, the sequence $\{a_n\}$ is unbounded. Therefore, for every $r \in \rr$, the interval $[0,r]$ contains only finitely many elements of $\{a_n\}$, say $a_1, \dots, a_k$ for $k \in \nn$. Since $\langle a_1, \dots, a_k \rangle \cap [0,r]$ is a finite set, it follows that $M \cap [0,r]$ is finite as well. Because $|[0,r] \cap M| < \infty$ for all $r \in \rr$, it follows that $M$ does not have any limit points in $\rr$.
		
		Now suppose that $N$ is a nontrivial submonoid of $M$. Notice that, being a subset of $M$, the monoid $N$ cannot have any limit points in $\rr$. Thus, by Lemma~\ref{lem:no limit point implies strongly increasing subset}, the set~$N$ is the underlying set of a strongly increasing sequence of rationals. Hence $N$ is a strongly increasing Puiseux monoid, and the direct implication follows.
		
		For the converse implication, suppose that $M$ is not strongly increasing. We will check that, in this case, $M$ contains a submonoid that is not increasing. If $M$ is not increasing, then $M$ is a submonoid of itself that is not increasing. Suppose, therefore, that $M$ is increasing. By Proposition~\ref{prop:atoms of increasing monoids}, the monoid $M$ is atomic, and we can list its atoms increasingly. Let $\{a_n\}$ be an increasing sequence with underlying set $\mathcal{A}(M)$. Because $M$ is not strongly increasing, there exists a positive real $\ell$ that is the limit of the sequence $\{a_n\}$. Since $\ell$ is a limit point of $M$, which is closed under addition, it follows that $2\ell$ and $3\ell$ are both limit points of $M$. Let $\{b_n\}$ and $\{c_n\}$ be sequences in~$M$ having infinite underlying sets such that $\lim b_n = 2\ell$ and $\lim c_n = 3\ell$. Furthermore, assume that for each $n \in \nn$,
		\begin{equation}
			|b_n - 2\ell| < \frac{\ell}4 \ \text{ and } \ |c_n - 3\ell| < \frac{\ell}4. \label{eq:decreasing is not hereditary}
		\end{equation}
		Take $N$ to be the submonoid of $M$ generated by the set $A := \{b_n,c_n \mid n \in \nn\}$. Note that $A$ contains at least two limit points. Let us verify that $N$ is atomic with $\mathcal{A}(N) = A$. The inequalities \eqref{eq:decreasing is not hereditary} immediately imply that $A$ is bounded from above by $3\ell + \ell/4$. On the other hand, proving that $\mathcal{A}(N) = A$ amounts to showing that the sets $A$ and $A+A$ are disjoint. To verify this, it suffices to note that
		\begin{align*}
			\inf (A + A) &= \inf \big\{b_m + b_n, b_m + c_n, c_m + c_n \mid m,n \in \nn \big\} \\
							  &\ge \min \bigg\{4\ell - \frac{\ell}{2}, \ 5\ell - \frac{\ell}{2}, \ 6\ell - \frac{\ell}{2} \bigg\} \\
							  &> 3\ell + \frac{\ell}{4} \ge \sup A.
		\end{align*}
		Thus, $\mathcal{A}(N) = A$. Since every increasing sequence has at most one limit point in $\rr$, the set $A$ cannot be the underlying set of an increasing rational sequence. As every generating set of $N$ contains $A$, we conclude that $N$ is not an increasing Puiseux monoid, which completes the proof.
	\end{proof}
	
	As a direct consequence of Theorem \ref{thm:strongly increasing iff super increasing}, one obtains the following corollary.
	
	\begin{cor}
		Being atomic, increasing, and strongly increasing are hereditary properties on the class of strongly increasing Puiseux monoids.
	\end{cor}

	\medskip
	\section{Decreasing Puiseux Monoids} \label{sec:decreasing PM}
	
	Now that we have explored the structure of increasing Puiseux monoids, we will focus on the study of their decreasing counterpart. If a Puiseux monoid is decreasing, then it is obviously bounded. On the other hand, there are bounded Puiseux monoids that are not even monotone; see Example~\ref{ex:bounded PM that is neither decreasing nor increasing}. However, every strongly bounded Puiseux monoid is decreasing, as we will show in Proposition~\ref{prop:SB PM are strongly decreasing}.
	
	By contrast to the results we obtained in the previous section, the next proposition will show that being decreasing is almost never hereditary. In fact, we prove that being decreasing is hereditary only on those Puiseux monoids that are isomorphic to numerical semigroups.
	
	\begin{lemma} \label{lem:decreasing monoids have infinite limit points}
		If $M$ is a nontrivial decreasing Puiseux monoid, then exactly one of the following conditions holds:
		\begin{enumerate}
			\item $M$ is isomorphic to a numerical semigroup;
			\vspace{3pt}
			\item $M$ contains infinitely many limit points in $\rr$.
		\end{enumerate}
	\end{lemma}
	
	\begin{proof}
		Suppose that $M$ is not isomorphic to a numerical semigroup. Since $M$ is not trivial, it fails to be finitely generated. Therefore it can be generated by a strictly decreasing sequence $\{a_n\}$. The sequence $\{a_n\}$ must converge to a non-negative real number $\ell$. Since $\{k a_n\} \subseteq M$ converges to $k\ell$ for every $k \in \nn$, if $\ell \neq 0$, then every element of the infinite set $\{k\ell \mid k \in \nn\}$ is a limit point of $M$. On the other hand, if $\ell = 0$, then every term of the sequence $\{a_n\}$ is a limit point of $M$; this is because for every fixed $k \in \nn$ the sequence $\{a_k + a_n\} \subseteq M$ converges to $a_k$. Hence $M$ has infinitely many limit points in $\rr$.
		
		Now let us verify that at most one of the above two conditions can hold. For this, assume that $M$ is isomorphic to a numerical semigroup. So $M$ is finitely generated, namely, $M = \langle r_1, \dots, r_n \rangle$, where $n \in \nn$ and $r_i \in \qq_{> 0}$ for $i=1,\dots, n$. For every $r \in \rr$ the interval $[0,r]$ contains only finitely many elements of $M$. Since $M \cap [0,r]$ is finite for all $r \in \rr$, it follows that $M$ cannot have any limit points in the real line.
	\end{proof}
	
	\begin{proposition}
		Let $M$ be a nontrivial decreasing Puiseux monoid. Then exactly one of the following conditions holds:
		\begin{enumerate}
			\item $M$ is isomorphic to a numerical semigroup;
			\vspace{3pt}
			\item $M$ contains a submonoid that is not decreasing.
		\end{enumerate}
	\end{proposition}
	
	\begin{proof}
		Suppose that $M$ is not isomorphic to a numerical semigroup. Let us construct a submonoid of $M$ that fails to be decreasing. Lemma~\ref{lem:decreasing monoids have infinite limit points} implies that $M$ has a nonzero limit point $\ell$. Since $M$ is closed under addition, $2\ell$ and $3\ell$ are both limit points of $M$. An argument as the one given in the proof of Theorem~\ref{thm:strongly increasing iff super increasing} will guarantee the existence of sequences $\{a_n\}$ and $\{b_n\}$ in $M$ having infinite underlying sets such that $\{a_n\}$ converges to $2\ell$, $\{b_n\}$ converges to $3\ell$, and the submonoid $N = \langle a_n, b_n \mid n \in \nn \rangle$ of $M$ is atomic with $\mathcal{A}(M) = \{a_n,b_n \mid n \in \nn\}$. Since every decreasing sequence of $\qq$ contains at most one limit point, $\mathcal{A}(M)$ cannot be the underlying set of a decreasing sequence of rationals. As every generating set of $N$ must contain $\mathcal{A}(M)$, we can conclude that $N$ is not decreasing. Hence at least one of the given conditions must hold.
		
		To see that both conditions cannot hold simultaneously, it suffices to observe that if~$M$ is isomorphic to a numerical semigroup, then every nontrivial submonoid of~$M$ is also isomorphic to a numerical semigroup and, therefore, decreasing.
	\end{proof}
	
	Similarly, as we did in the case of increasing Puiseux monoids, we will split the family of decreasing Puiseux monoids into two fundamental subfamilies, depending on whether $0$ is or is not a limit point. We say that a non-negative sequence of rationals is \emph{strongly decreasing} if it is decreasing and it converges to zero. A non-negative decreasing sequence of rationals converging to a positive real is called \emph{weakly decreasing}.
	
	\begin{definition}
		A Puiseux monoid is \emph{strongly decreasing} if it can be generated by a strongly decreasing sequence of rational numbers. On the other hand, a Puiseux monoid is said to be \emph{weakly decreasing} if it can be generated by a weakly decreasing sequence of rationals.
	\end{definition}
	
	Observe that if a Puiseux monoid $M$ is weakly decreasing, then it has a generating sequence decreasing to a positive real number and, therefore, $0$ is not in the closure of~$M$. As a consequence, every weakly decreasing Puiseux monoid must be atomic. The next proposition describes those Puiseux monoids that are both strongly and weakly decreasing.
	
	\begin{proposition}
		A decreasing Puiseux monoid is either strongly or weakly decreasing. A Puiseux monoid is both strongly and weakly decreasing if and only if it is isomorphic to a numerical semigroup.
	\end{proposition}
	
	\begin{proof}
		As in the case of increasing Puiseux monoids, the first statement follows immediately. Now suppose that $M$ is a Puiseux monoid that is both strongly and weakly decreasing. Since $M$ is weakly decreasing, $0$ is not a limit point of $M^\bullet \!$. Let $\{a_n\}$ be a sequence decreasing to zero such that $M = \langle a_n \mid n \in \nn \rangle$. Because $0$ is not a limit point of $M^\bullet$, there exists $n_0 \in \nn$ such that $a_n = 0$ for all $n \ge n_0$. Hence $M$ is isomorphic to a numerical semigroup. As in the increasing case, it is easily seen that every numerical semigroup is both strongly and weakly decreasing.
	\end{proof}
	
	We mentioned at the beginning of this section that every strongly bounded Puiseux monoid is decreasing. Indeed, a stronger statement holds. 
	
	\begin{proposition} \label{prop:SB PM are strongly decreasing}
		Every strongly bounded Puiseux monoid is strongly decreasing.
	\end{proposition}
	
	\begin{proof}
		Let $M$ be a strongly bounded Puiseux monoid. Since the trivial monoid is both strongly bounded and strongly decreasing, for this proof we will assume that $M \neq \{0\}$. Let $S \subset\qq_{> 0}$ be a generating set of $M$ such that $\mathsf{n}(S)$ is bounded. Since $\mathsf{n}(S)$ is finite, we can take $m$ to be the least common multiple of the elements of $\mathsf{n}(S)$. The map $x \mapsto \frac{1}{m}x$ is an order-preserving isomorphism from $M$ to $M' = \frac{1}{m}M$. Consequently, $M'$ is strongly decreasing if and only if $M$ is strongly decreasing. In addition, $S' = \frac{1}{m}S$ generates $M'$. Since $\mathsf{n}(S') = \{1\}$, it follows that $S'$ is the underlying set of a strongly decreasing sequence of rationals. Hence $M'$ is a strongly decreasing Puiseux monoid, which implies that $M$ is strongly decreasing as well.
	\end{proof}
	
	Recall that a Puiseux monoid $M$ is finite if $\pval(\mathsf{d}(M^\bullet)) = \{0\}$ for all but finitely many primes $p$. Strongly decreasing Puiseux monoids are not always strongly bounded, even if we require them to be finite. For example, if $r \in \qq$ such that $0 < r < 1$ and both $\mathsf{n}(r)$ and $\mathsf{d}(r)$ are different from $1$, then the Puiseux monoid $M_r = \langle r^n \mid n \in \nn \rangle$ is atomic and $\mathcal{A}(M_r) = \{r^n \mid n \in \nn\}$ (this will be proved in Theorem~\ref{thm:atomic classification of multiplicative cyclic Puiseux monoids}). As a result, $M_r$ is finite and strongly decreasing. However, $M_r$ fails to be strongly bounded. On the other hand, not every bounded Puiseux monoid is decreasing, as illustrated in Example~\ref{ex:bounded PM that is neither decreasing nor increasing}.
	
	Because numerical semigroups are finitely generated, they are both increasing and decreasing Puiseux monoids. We end this section showing that numerical semigroups are the only such Puiseux monoids.
	
	\begin{proposition}
		A nontrivial Puiseux monoid $M$ is isomorphic to a numerical semigroup if and only if $M$ is both increasing and decreasing.
	\end{proposition}
	
	\begin{proof}
		If $M$ is isomorphic to a numerical semigroup, then it is finitely generated and, consequently, increasing and decreasing.
		
		Conversely, suppose that $M$ is a nontrivial Puiseux monoid that is increasing and decreasing. Proposition~\ref{prop:atoms of increasing monoids} implies that $M$ is atomic and, moreover, $\mathcal{A}(M)$ is the underlying set of an increasing sequence (because $\mathcal{A}(M) \neq \emptyset$). Suppose, by way of contradiction, that $\mathcal{A}(M)$ is not finite. In this case, $\mathcal{A}(M)$ does not contain a largest element. Since $M$ is decreasing, there exists $D = \{d_n \mid n \in \nn \} \subset \qq_{> 0}$ such that $d_1 > d_2 > \cdots$ and $M = \langle D \rangle$. Let $m = \min \{n \in \nn \mid d_n \in \mathcal{A}(M)\}$, which must exist because $\mathcal{A}(M) \subseteq D$. Since $\mathcal{A}(M)$ is contained in $D$, the minimality of $m$ implies that~$d_m$ is the largest element of $\mathcal{A}(M)$, which is a contradiction. Hence $\mathcal{A}(M)$ is finite. Since $M$ is atomic and $\mathcal{A}(M)$ is finite, $M$ is isomorphic to a numerical semigroup.
	\end{proof}

	\medskip
	\section{Prime Reciprocal Puiseux Monoids} \label{sec:primary Puiseux Monoids}
	
	In this section, we take a step further in our search of non-finitely generated atomic Puiseux monoids. We investigate the atomic structure of submonoids of those Puiseux monoids that can be generated by reciprocals of primes. Observe that such Puiseux monoids form a special subclass of strongly decreasing Puiseux monoids.
	
	\begin{definition}
		A Puiseux monoid $M$ is said to be \emph{prime reciprocal}\footnote{In the original paper, as published in Semigroup Forum, instead of ``prime reciprocal" the term``primary" was used. However, we have decided to use the former term in this updated arXiv version as the term ``primary" generates some conflict of names with a different class of monoids already studied under the same name.} if there exists a set $P$ of primes such that $M = \langle 1/p \mid p \in P \rangle$.
	\end{definition}
	
	Let $M$ be a Puiseux monoid, and let $N$ be a submonoid of $M$. If $M$ is finitely generated, then $N$ is also finitely generated. Thus, being finitely generated is hereditary on the class of finitely generated Puiseux monoids. As we should expect, not every property of a Puiseux monoid is inherited by its submonoids. For example, being antimatter is not hereditary on the class of antimatter Puiseux monoids; for instance, $\qq_{\ge 0}$ is antimatter because every positive rational can be expressed as the addition of two positive rationals, but it contains the atomic submonoid $\nn_0$, which satisfies $\mathcal{A}(\nn_0) = \{1\}$. Moreover, as Corollary~\ref{cor:strongly bounded is not hereditary on primary PM} indicates, boundedness and strong boundedness are not hereditary, even on the class of prime reciprocal Puiseux monoids.
	
	Let $S$ be a set of naturals. If the series $\sum_{s \in S} 1/s$ diverges, $S$  is said to be \emph{substantial}. If $S$ is not substantial, it is said to be \emph{insubstantial} (see~\cite{pC}). For example, it is well known that the set of prime numbers is substantial as it was first noticed by Euler that the series of reciprocal primes is divergent.
	
	\begin{proposition} \label{prop:strongly bounded PM having a submonoid that is not even bounded}
		Let $P$ be a set of primes, and let $M$ be the prime reciprocal Puiseux monoid $\langle 1/p \mid p \in P \rangle$. If every submonoid of $M$ is bounded, then $P$ is insubstantial.
	\end{proposition}
		
	\begin{proof}
		Suppose, by way of contradiction, that $P$ is substantial. Then $P$ must contain infinitely many primes. Let $\{p_n\}$ be a strictly increasing enumeration of the elements in $P$. Take $N$ to be the submonoid of $M$ generated by $A = \{a_n \mid n \in \nn\}$, where
		\[
			a_n = \sum_{i=1}^n \frac{1}{p_i}.
		\]
		Since $P$ is substantial, $A$ is unbounded. We will show that $N$ fails to be bounded. For this purpose, we verify that $\mathcal{A}(N) = A$, which implies that every generating set of $N$ contains $A$ and, therefore, must be unbounded. Suppose that
		\begin{equation} \label{eq:boundedness in not hereditary 1}
			a_n = a_{n_1} + \dots + a_{n_{\ell}}
		\end{equation}
		for some $\ell, n, n_1, \dots, n_\ell \in \nn$ such that $n_1 \le \dots \le n_\ell$. Since $\{a_n\}$ is an increasing sequence, $n \ge n_\ell$. After multiplying the equation \eqref{eq:boundedness in not hereditary 1} by $m = p_1 \dots p_n$ and moving every summand but $m/p_n$ to the right-hand side, we obtain
		\begin{equation} \label{eq:boundedness in not hereditary 2}
			p_1 \cdots p_{n-1} = \sum_{j=1}^\ell \sum_{i=1}^{n_j} \frac{m}{p_i} - \sum_{i=1}^{n-1} \frac{m}{p_i}.
		\end{equation}
		Now we observe that if $n$ were strictly greater than $n_\ell$, then $m/p_i$ would be an integer divisible by $p_n$ for each $i = 1, \dots, n_j$ and $j = 1, \dots, \ell$, which would imply that the right-hand side of \eqref{eq:boundedness in not hereditary 2} is divisible by $p_n$. This cannot be possible because $p_1 \cdots p_{n-1}$ is not divisible by $p_n$. Thus, $n = n_\ell$ and so $\ell = 1$. Since $\{a_n\}$ is an increasing sequence satisfying that $a_n \notin \langle a_1, \dots, a_{n-1} \rangle$, Proposition~\ref{prop:atoms of increasing monoids} ensures that $\mathcal{A}(N) = A$. As a result, $M$ contains a submonoid that fails to be bounded; but this is a contradiction. Hence the set $P$ is insubstantial.
	\end{proof}
	
	For $m,n \in \nn_0$ such that $n > 0$ and $\gcd(m,n) = 1$, Dirichlet's theorem states that the set $P$ of all primes $p$ satisfying that $p \equiv m \pmod n$ is infinite. For a relatively elementary proof of Dirichlet's theorem, see \cite{hS50}. Furthermore, it is also known that the set $P$ is substantial; indeed, as indicated in \cite[page 156]{tA76}, there exists a constant~$A$ for which
	\begin{equation} \label{eq:prime in arithmetic progression are substantial}
		\sum_{p \in P, p \le x} \frac{1}{p} = \frac{1}{\varphi(n)} \log \log x + A + O\bigg(\frac{1}{\log x}\bigg),
	\end{equation}
	where $\varphi$ is the Euler totient function. In particular, the set comprising all primes of the form $4k+1$ (or $4k+ 3$) is substantial. The next corollary follows immediately from Proposition~\ref{prop:strongly bounded PM having a submonoid that is not even bounded} and equation~\eqref{eq:prime in arithmetic progression are substantial}.
	
	\begin{cor} \label{cor:strongly bounded is not hereditary on primary PM}
		 Let $m,n \in \nn_0$ such that $n > 0$ and $\gcd(m,n) = 1$, and let $P$ be the set of all primes $p$ satisfying $p \equiv m \pmod n$. Then the prime reciprocal Puiseux monoid $M = \langle 1/p \mid p \in P \rangle$ contains an unbounded submonoid.
	\end{cor}
	
	We say that a monoid $M$ is \emph{hereditarily atomic} if each submonoid of $M$ is atomic, i.e., being atomic is an hereditary property on $M$. Numerical semigroups and strongly increasing Puiseux monoids are hereditarily atomic. More generally, if $M$ is a Puiseux monoid not having $0$ as a limit point, then no submonoid of $M$ has $0$ as a limit point and, as a consequence, $M$ is hereditarily atomic. According to Theorem~\ref{thm:super atomicity of a prime-generated Puiseux monoid}, every prime reciprocal Puiseux monoid is hereditarily atomic.
	
	Let $P$ be a set of primes, and let $r \in \qq_{> 0}$. We denote by $\mathsf{D}_P(r)$ the set of primes $p \in P$ dividing $\mathsf{d}(r)$. Besides, if $R \subseteq \qq_{> 0}$, then we set $\mathsf{D}_P(R) = \cup_{r \in R} \mathsf{D}_P(r)$. The following lemma is used in the proof of Theorem~\ref{thm:super atomicity of a prime-generated Puiseux monoid}.
	
	\begin{lemma} \label{lem:properties of D}
		Let $P$ be a set of primes and, for $n \in \nn$, let $r, r_1, \dots, r_n$ be positive rationals such that $r = r_1 + \dots + r_n$. Then $\mathsf{D}_P(r) \subseteq \mathsf{D}_P(r_1) \cup \dots \cup \mathsf{D}_P(r_n)$.
	\end{lemma}
	
	\begin{proof}
		Take $p \in \mathsf{D}_P(r)$. Then $p$ is a prime in $P$ dividing $\mathsf{d}(r)$. Multiplying the equation $r = r_1 + \dots + r_n$ by $d = \mathsf{d}(r)\mathsf{d}(r_1) \cdots \mathsf{d}(r_n)$, we get
		\begin{equation} \label{eq:lemma properties of D}
			\mathsf{d}(r_1) \cdots \mathsf{d}(r_n) \mathsf{n}(r) = \sum_{i=1}^n m_i \mathsf{n}(r_i),
		\end{equation}
		where $m_i = d/\mathsf{d}(r_i)$ for every $i \in \{1,\dots,n\}$. Since $p$ divides each summand on the right-hand side of equation \eqref{eq:lemma properties of D}, it must divide $\mathsf{d}(r_i)$ for some $i \in \{1,\dots,n\}$. Therefore $p \in \mathsf{D}_P(r_i)$, and the desired set inclusion follows.
	\end{proof}
	
	\begin{theorem} \label{thm:super atomicity of a prime-generated Puiseux monoid}
		Every prime reciprocal Puiseux monoid is hereditarily atomic.
	\end{theorem}
	
	\begin{proof}
		Let $P$ be a set of primes, and let $M$ be the prime reciprocal Puiseux monoid generated by the set $\{1/p \mid p \in P \}$. If $P$ is finite, then $M$ is isomorphic to a numerical semigroup, and so every submonoid of $M$ is atomic. So we assume that $P$ contains infinitely many primes. Let $p_1, p_2, \dots$ be an increasing enumeration of the elements in~$P$. First, we show that for all $x \in M^\bullet$ there exist only finitely many $N \in \nn$ such that
		\begin{equation} \label{eq:super atomicity of a prime-generated Puiseux monoid 1}
			x = \sum_{i=1}^N \alpha_i \frac{1}{p_{n_i}},
		\end{equation}
		for some $\alpha_1, \dots, \alpha_N, n_1, \dots, n_N \in \nn$ with $n_1 < \dots < n_N$. Since $\{1/p \mid p \in P\}$ generates~$M$, there exists at least one natural $N_0$ such that equation \eqref{eq:super atomicity of a prime-generated Puiseux monoid 1} holds. Let us check that if a natural $N$ satisfies \eqref{eq:super atomicity of a prime-generated Puiseux monoid 1}, then $N \le x + N_0$. Suppose, by way of contradiction, that $N$ is a natural number greater than $x + N_0$ and
		\begin{equation} \label{eq:super atomicity of a prime-generated Puiseux monoid 2}
			x = \sum_{j=1}^N \beta_j \frac{1}{p_{m_j}}
		\end{equation}
		for some $\beta_1, \dots, \beta_N, m_1, \dots, m_N \in \nn$ with $m_1 < \dots < m_N$. Equation~\eqref{eq:super atomicity of a prime-generated Puiseux monoid 2} forces the cardinality of the set $\{j \in \{1,\dots, N\} \mid p_{m_j} \text{ divides} \ \beta_j\}$ to be at most $\lfloor x \rfloor$. Since $N > x + N_0 \ge \lfloor x \rfloor + N_0$, there exists $k \in \{1,\dots,N\}$ such that $p_{m_k} \notin \{p_{n_1}, \dots, p_{n_{N_0}}\}$ and $p_{m_k} \nmid \beta_k$. After equating both right-hand sides of equations \eqref{eq:super atomicity of a prime-generated Puiseux monoid 1} and \eqref{eq:super atomicity of a prime-generated Puiseux monoid 2}, and multiplying the resulting equality by $q = p_{n_1} \cdots p_{n_{N_0}} p_{m_1} \cdots, p_{m_N}$, one obtains
		\begin{equation} \label{eq:super atomicity of a prime-generated Puiseux monoid 3}
			\beta_k B_k + \sum_{j=1, j \neq k}^N \beta_j B_j - \sum_{i=1}^{N_0} \alpha_i A_i = 0,
		\end{equation}
		where $A_i = q/p_{n_i}$ for $i = 1,\dots, N_0$ and $B_j = q/p_{m_j}$ for $j = 1,\dots, N$. Note that every summand in \eqref{eq:super atomicity of a prime-generated Puiseux monoid 3} except the first one is divisible by $p_{m_k}$. But this is a contradiction and, therefore, every $N$ satisfying \eqref{eq:super atomicity of a prime-generated Puiseux monoid 1} is less than or equal to $x + N_0$. Hence there are only finitely many $N \in \nn$ satisfying \eqref{eq:super atomicity of a prime-generated Puiseux monoid 1}.
		
		Now, we are in a position to prove that every submonoid of $M$ is atomic. Let us assume, by way of contradiction, that $M$ contains a non-atomic submonoid $M'$. Fix $z \in M' \setminus \langle \mathcal{A}(M') \rangle$. Let $n \in \nn$ such that there exist $x_1, \dots, x_n \in M'^\bullet$ for which $z = x_1 + \dots + x_n$. Set $D = \mathsf{D}_P(x_1) \cup \dots \cup \mathsf{D}_P(x_n)$. Since $D$ contains only finitely many primes, the set
		\[
			I = \bigg\{n \in \nn \ \bigg| \ \exists \ r_1, \dots, r_n \in M'^\bullet : z = \sum_{i=1}^n r_i \ \text{ and } \ \bigcup_{i=1}^n \mathsf{D}_P(r_i) \subseteq D \bigg\}
		\]
		is finite. Take $m$ to be the maximum of $I$, and take $r_1, \dots, r_m \in M'^\bullet$ such that $z = r_1 + \dots + r_m$ and $\mathsf{D}_P(r_1) \cup \dots \cup \mathsf{D}_P(r_m) \subseteq D$. Since $z \notin \langle \mathcal{A}(M') \rangle$, there is an element of $\{r_1, \dots, r_m\}$, say $r_m$ without loss, that is not an atom of $M'$. Take $k \in \zz_{\ge 2}$ and $r'_1, \dots, r'_k \in M'^\bullet$\! so that $r_m = r'_1+ \dots + r'_k$. By the maximality of $m$, there exists $j \in \{1,\dots,k\}$ for which $\mathsf{D}_P(r_j')$ fails to be a subset of $D$. On the other hand, Lemma \ref{lem:properties of D} guarantees that $\mathsf{D}_P(r_m) \subseteq \mathsf{D}_P(r_1') \cup \dots \cup \mathsf{D}_P(r_k')$. Therefore
		\[
			|\mathsf{D}_P\big(\{r_1, \dots, r_m\}\big)| < |\mathsf{D}_P\big(\{r_1, \dots, r_{m-1}\} \cup \{r'_1, \dots, r'_k\}\big)|.
		\]
		So for every $N \in \nn$ there is a natural $n$ with $z = r_1 + \dots + r_n$ for some $r_1, \dots, r_n \in M'^\bullet$ such that $|\mathsf{D}_P(\{r_1, \dots, r_n\})| > N$. Writing each $r_j$ in $z = r_1 + \dots + r_n$ as a sum of elements in $\{1/p \mid p \in P\}$, we would be able to write $z$ as in equation \eqref{eq:super atomicity of a prime-generated Puiseux monoid 1} for infinitely many $N \in \nn$, which is a contradiction. Hence every submonoid of $M$ is atomic, i.e.,~$M$ is hereditarily atomic.
	\end{proof}

	The following corollary is an immediate consequence of Theorem~\ref{thm:super atomicity of a prime-generated Puiseux monoid}.
	
	\begin{cor}
		If $S \subseteq \qq_{> 0}$ satisfies that $\mathsf{d}(s)$ is either $1$ or prime for all $s \in S$, then the Puiseux monoid generated by $S$ is hereditarily atomic.
	\end{cor}

	\medskip
	\section{Multiplicatively Cyclic Puiseux Monoids} \label{sec:Multiplicative Cyclic Puiseux Monoids}
	
	We know that finitely generated Puiseux monoids are isomorphic to numerical semigroups. It is natural to wonder which are the simplest families of Puiseux monoids that are not isomorphic to numerical semigroups. Since the members of such families must be infinitely generated, it would be convenient to look into classes of Puiseux monoids infinitely generated by well-behaved sequences of rationals.
	
	Numerical semigroups generated by intervals, arithmetic sequences, and generalized arithmetic sequences have been intensely studied (see \cite{GR99,ACHP07,gM04} and the references therein). In particular, the simplicity of arithmetic sequences has facilitated the exploration of the combinatorial and algebraic structure of the numerical semigroups they generate as well as the factorization invariants such numerical semigroups exhibit. We may want to study, in principle, the family of Puiseux monoids generated by arithmetic sequences (which happen to be increasing). However, notice that if a Puiseux monoid $M$ is generated by an arithmetic sequence $\{r +ns\}$, where $r,s \in \qq_{> 0}$, then the map $x \mapsto \mathsf{d}(r)\mathsf{d}(s)x$ defines an isomorphism from $M$ onto the numerical semigroup $\langle \mathsf{d}(s)\mathsf{n}(r) + n \mathsf{d}(r) \mathsf{n}(s) \mid n \in \nn \rangle$. So Puiseux monoids generated by arithmetic sequences are isomorphic to numerical semigroups.
	
	By contrast, Puiseux monoids generated by geometric sequences are not necessarily finitely generated. For example, if $z \in \zz_{\ge 2}$, the Puiseux monoid $M = \langle 1/z^n \mid n \in \nn \rangle$ is antimatter, and so it fails to be finitely generated. We might expect that the controlled behavior of a rational geometric sequence leads us to a better understanding of the atomicity and boundedness of the Puiseux monoid it generates. In this section, we will explore the atomicity and boundedness of those Puiseux monoids that can be generated by geometric sequences.
	
	\begin{definition}
		For $r \in \qq_{> 0}$, we call the Puiseux monoid generated by the positive powers of $r$ the \emph{multiplicatively $r$-cyclic} (or just \emph{multiplicatively cyclic}) and we denote it by $M_r$, that is, $M_r = \langle r^n \mid n \in \nn \rangle$.
	\end{definition} 
	
	\noindent \textbf{Remarks:} Note that Puiseux monoids of the form $\langle ar^n \mid n \in \nn \rangle$ for $a,r \in \qq_{> 0}$ are not more general than those we defined as multiplicatively $r$-cyclic; this is because multiplication by $a$ gives an isomorphism of Puiseux monoids. Similarly, the Puiseux monoid $\langle r^n \mid n \in \nn_0 \rangle$ is isomorphic to the multiplicatively cyclic monoid $M_r$. However, note that $\langle r^n \mid n \in \nn_0 \rangle$ is a cyclic rational semiring with identity, while $M_r$ does not contain an identity as a rational semiring. We emphasize that when dealing with $M_r$, we will only be interested in its algebraic structure with respect to the addition, so we shall study its atomicity as an additive monoid.
	
	\vspace{5pt}
	
	We recall that an antimatter monoid is a monoid having no atoms. The next theorem describes the sets of atoms of multiplicatively cyclic Puiseux monoids, indicating, in particular, which of these monoids are atomic.
	
	\begin{theorem} \label{thm:atomic classification of multiplicative cyclic Puiseux monoids}
		For $r \in \qq_{> 0}$, let $M_r$ be the multiplicatively $r$-cyclic Puiseux monoid. Then the following statements hold.
		\begin{itemize}
			\item If $\mathsf{d}(r)=1$, then $M_r$ is atomic with $\mathcal{A}(M_r) = \{\mathsf{n}(r)\}$.
			\vspace{3pt}
			\item If $\mathsf{d}(r) > 1$ and $\mathsf{n}(r) = 1$, then $M_r$ is antimatter.
			\vspace{3pt}
			\item If $\mathsf{d}(r) > 1$ and $\mathsf{n}(r) > 1$, then $M_r$ is atomic with $\mathcal{A}(M_r) = \{r^n \mid n \in \nn\}$.
		\end{itemize}
	\end{theorem}
	
	\begin{proof}
		Set $a = \mathsf{n}(r)$ and $b = \mathsf{d}(r)$. If $b = 1$, then $M_r = \langle a^n \mid n \in \nn \rangle = \langle a \rangle$, which immediately implies that $M_r$ is atomic and $\mathcal{A}(M_r) = \{a\}$. Suppose now that $a = 1$ and $b > 1$. In this case, $M_r = \langle 1/b^n \mid n \in \nn \rangle$. Since $1/b^n = b(1/b^{n+1})$ for every $n \in \nn$, it follows that $M_r$ is antimatter.
		
		Set $R = \{r^n \mid n \in \nn\}$. We proceed to show the last statement, that is, the case where $a > 1$ and $b > 1$. First, we argue the case $a > b$. Since $\{r^n\}$ is an increasing sequence generating $M_r$, by Proposition \ref{prop:atoms of increasing monoids} the monoid $M_r$ is atomic. Let us find $\mathcal{A}(M_r)$. Take $n \in \nn$ such that $n > 1$, and suppose that there exist $k, c_k \in \nn$ and $c_i \in \nn_0$ for every $i = 1,\dots,k-1$ satisfying
		\begin{equation} \label{eq:multiplicative cyclic 1}
			\frac{a^n}{b^n} = c_1\frac{a}{b} + \dots + c_k\frac{a^k}{b^k}.
		\end{equation}
		If $k < n$, then after multiplying equation \eqref{eq:multiplicative cyclic 1} by $b^n$, it can be easily seen that every prime divisor of $b$ must divide $a$, which is not possible because $\gcd(a,b) = 1$. Therefore $r^n \notin \langle r, \dots, r^{n-1} \rangle$ for any $n \in \nn$. By Proposition~\ref{prop:atoms of increasing monoids}, one has $\mathcal{A}(M_r) = R$.
		
		Finally, suppose $a < b$. Take $n \in \nn$, and write
		\begin{equation} \label{eq:multiplicative cyclic 2}
			\frac{a^n}{b^n} = c_n\frac{a^n}{b^n} + \dots + c_{n+k}\frac{a^{n+k}}{b^{n+k}},
		\end{equation}
		where $k \in \nn_0$ and $c_i \in \nn_0$ for every $i = n, \dots, n+k$. Notice that $c_n \in \{0,1\}$. Suppose, by way of contradiction, that $c_n = 0$. In this case, $k \ge 1$. Let $p$ be a prime dividing~$a$, and let $\alpha$ be the maximum power of $p$ dividing $a$. Applying the $p$-adic valuation function to equation \eqref{eq:multiplicative cyclic 2}, one obtains
		\begin{equation} \label{eq:multiplicative cyclic 3}
			\alpha n = \pval\bigg(\frac{a^n}{b^n}\bigg) = \pval\bigg(\sum_{i=1}^k c_{n+i} \frac{a^{n+i}}{b^{n+i}}\bigg) \ge \min_{1 \le i \le k}\bigg\{\pval\bigg(c_{n+i} \frac{a^{n+i}}{b^{n+i}}\bigg)\bigg\} \ge \alpha(n+m),
		\end{equation}
		where $m = \min\{i \in \{1, \dots, k\} \mid c_{n+i} \neq 0\}$. Inequality \eqref{eq:multiplicative cyclic 3} yields a contradiction because $m \ge 1$. Therefore $c_n = 1$, and so $c_{n+i} = 0$ for every $i \ge 1$. Since $(a/b)^n$ cannot be expressed in a nontrivial way as a sum of elements in $R$, one finds that $(a/b)^n$ is an atom. Hence $R$ is the set of atoms of $M_r$ and, as a result, $M_r$ is atomic.
	\end{proof}
	
	With notation as in Theorem~\ref{thm:atomic classification of multiplicative cyclic Puiseux monoids}, if $\mathsf{n}(r) = 1$ or $\mathsf{d}(r)=1$, then the multiplicatively cyclic Puiseux monoid $M_r$ is strongly bounded. If $\mathsf{n}(r), \mathsf{d}(r) > 1$ and $r < 1$, then $M_r$ is bounded. However $M_r$ cannot be strongly bounded because every generating set of $M_r$ must contain the set $R = \{r^n \mid n \in \nn\}$, which is not strongly bounded. By a similar argument, $M_r$ is not bounded when $\mathsf{n}(r), \mathsf{d}(r) > 1$ and $r > 1$.
	
	\begin{cor}
		For $r \in \qq_{> 0}$, let $M_r$ be the multiplicatively $r$-cyclic Puiseux monoid. Then the following statements hold.
		\begin{itemize}
			\item If $\mathsf{n}(r)=1$ or $\mathsf{d}(r)=1$, then $M_r$ is strongly bounded. \vspace{3pt}
			\item If $\mathsf{n}(r), \mathsf{d}(r) > 1$ and $r < 1$, then $M_r$ is bounded but not strongly bounded.
			\vspace{3pt}
			\item If $\mathsf{n}(r), \mathsf{d}(r)  > 1$ and $r > 1$, then $M_r$ is not bounded.
		\end{itemize}
	\end{cor}
	
	As illustrated by Corollary~\ref{cor:strongly bounded is not hereditary on primary PM}, being bounded (or strongly bounded) is not hereditary on the class of prime reciprocal Puiseux monoids. Additionally, boundedness (resp., strong boundedness) is not hereditary on the class of bounded (resp., strongly bounded) multiplicatively cyclic Puiseux monoids.
	
	\begin{example}
		Let $M$ be the multiplicatively (1/2)-cyclic Puiseux monoid, that is, $M = \langle 1/2^n \mid n \in \nn \rangle$. It is strongly bounded, and yet its submonoid
		\begin{equation} \label{eq:strong boundedness is not hereditary even in multiplicatively cyclic PM}
			N = \bigg\langle \sum_{i=1}^n \frac{1}{2^i} \ \bigg{|} \ n \in \nn \bigg\rangle = \bigg\langle \frac{2^n - 1}{2^n} \ \bigg{|} \ n \in \nn \bigg\rangle
		\end{equation}
		is not strongly bounded; to see this, it is enough to verify that $\mathcal{A}(N) = S$, where $S$ is the generating set defining $N$ in \eqref{eq:strong boundedness is not hereditary even in multiplicatively cyclic PM}. Note that the sum of any two elements of the generating set $S$ is at least one, while every element of $S$ is less than one. Therefore each element of $S$ must be an atom of $N$, and so $\mathcal{A}(N) = S$.
	\end{example}

	We conclude this paper by showing that boundedness (resp., strong boundedness) is almost never hereditary on the class of bounded (resp., strongly bounded) multiplicatively cyclic Puiseux monoids.
	
	\begin{proposition}
		For $r \in \qq_{> 0}$, let $M_r$ be the multiplicatively $r$-cyclic Puiseux monoid. Then every submonoid of $M$ is bounded (or strongly bounded) if and only if $M_r$ is isomorphic to a numerical semigroup.
	\end{proposition}
	
	\begin{proof}
		Let $a$ and $b$ denote $\mathsf{n}(r)$ and $\mathsf{d}(r)$, respectively. To prove the direct implication, suppose, by way of contradiction, that $M_r$ is not isomorphic to a numerical semigroup. In this case, $b > 1$. Consider the submonoid $N = \langle s_1, s_2, \dots \rangle$ of $M$, where
		\[
			s_n = \frac{(nb^n + 1)a^n}{b^n}
		\]
		for every natural $n$. Proving the forward implication amounts to verifying that $N$ is not bounded and, as a consequence, not strongly bounded. First, let us check that $\mathcal{A}(N) = \{s_n \mid n \in \nn\}$. Note that
		\[
			s_{n+1} = \frac{((n+1)b^{n+1} + 1)a^{n+1}}{b^{n+1}} > \frac{(nb^{n+1} + b)a^n}{b^{n+1}} = s_n
		\]
		for each $n \in \nn$, and so $\{s_n\}$ is an increasing sequence. Moreover, it is easy to see that $s_n > n$ for every $n$. Thus, $\{s_n\}$ is unbounded. Suppose that there exist $k, \alpha_k \in \nn$ and $\alpha_i \in \nn_0$ for every $i = 1, \dots, k-1$ such that $s_n = \alpha_1s_1 + \dots + \alpha_k s_k$. Since $\{s_n\}$ is increasing and $\alpha_k > 0$, we have $k \le n$. Let $p$ be a prime divisor of $b$, and let $m = \pval(b)$. The fact that $p \nmid (nb^n + 1)a^n$ for every natural $n$ implies $\pval(s_n) = -mn$. Therefore
		\[
			-mn = \pval(s_n) \ge \min_{1 \le i \le k}\{\pval(\alpha_is_i)\} \ge \min_{1 \le i \le k}\{\pval(s_i)\} = -mk,
		\]
		which implies that $k \ge n$. Thus, $k=n$ and then $\alpha_1 = \dots = \alpha_{n-1} = 0$ and $\alpha_n = 1$. So $s_n \notin \langle s_1, \dots, s_{n-1} \rangle$ for every $n \in \nn$ and, by Proposition~\ref{prop:atoms of increasing monoids}, $\mathcal{A}(N) = \{s_n \mid n \in \nn\}$. Since $\mathcal{A}(N)$ is unbounded, $N$ cannot be a bounded Puiseux monoid.
		
		On the other hand, if $M_r$ is isomorphic to a numerical semigroup, then it is finitely generated and, hence, bounded and strongly bounded. This gives us the converse implication.
	\end{proof}
	
	\section*{Acknowledgments}
	
	While working on this paper, the first author was supported by the UC Berkeley Chancellor Fellowship, and the second author was under the University of Florida Mathematics Department Fellowship. With pleasure both authors thank Scott Chapman and the anonymous referee for valuable suggestions which resulted in an improvement of this paper.

\end{document}